\documentclass[a4,12pt]{amsart}
\usepackage{amsfonts}
\usepackage{amssymb}
\usepackage{amsmath}
\usepackage{amsthm}
\usepackage{fullpage}

\theoremstyle{plain}
\newtheorem{thm}{Theorem}[section]
\newtheorem{lem}[thm]{Lemma}
\theoremstyle{definition}

\newtheorem{ex}[thm]{Example}
\newtheorem{rem}[thm]{Remark}

\numberwithin{equation}{section}

\DeclareMathOperator{\Id}{Id}
\begin{document}
\title[Elliptic systems with nonlocal BCs]{Nonzero radial solutions for a class of elliptic systems with nonlocal BCs on annular domains}\thanks{Partially supported by G.N.A.M.P.A. - INdAM (Italy)}
\author{Gennaro Infante}%
\address{Gennaro Infante, Dipartimento di Matematica e Informatica, Universit\`{a} della
Calabria, 87036 Arcavacata di Rende, Cosenza, Italy}
\email{gennaro.infante@unical.it}
\author{Paolamaria Pietramala}%
\address{Paolamaria Pietramala, Dipartimento di Matematica e Informatica, Universit\`{a}
della Calabria, 87036 Arcavacata di Rende, Cosenza, Italy}
\email{pietramala@unical.it}
\begin{abstract} 
We provide new results on the existence, non-existence, localization and multiplicity of nontrivial solutions for systems of Hammerstein integral equations. 
Some of the criteria involve a comparison with the spectral radii of some associated linear operators. We apply our results to prove the existence of multiple nonzero radial solutions for some systems of elliptic boundary value 
problems subject to nonlocal boundary conditions. Our approach is topological and relies on the classical fixed point index. We present an example to illustrate our theory. 
\end{abstract}
\maketitle
\section{Introduction}
In the interesting paper~\cite{dolo}, 
Do \'O, Lorca and Ubilla, motivated by the work of Lee~\cite{lee} and by their previous paper~\cite{dolo0}, 
considered the existence of three positive solutions for the semilinear elliptic system 
\begin{gather}
\begin{aligned}\label{ellsyst}
\Delta u + \tilde{f}_1(|x|,u,v)=0,\ |x|\in [R_1,R_0], \\
\Delta v +\tilde{f}_2(|x|, u,v)=0,\, |x|\in [R_1,R_0],%
\end{aligned}
\end{gather}
subject to the \emph{non-homogenous} boundary conditions (BCs)
\begin{gather}
\begin{aligned}\label{nh-BCs}
u|_{\partial B_{R_1}}=0\ \text{and}\ u|_{\partial B_{R_0}}=A_1, \\
v|_{\partial B_{R_1}}=0\ \text{and}\ v|_{\partial B_{R_0}}=A_2, 
\end{aligned}
\end{gather}
where $x\in \mathbb{R}^n $, $0<R_1<R_0<\infty$, $A_1, A_2>0$ and $B_\rho=\{ x\in \mathbb{R}^n : |x|<\rho\}$. 
The methodology used in~\cite{dolo} is to seek radial solutions of the system 
\eqref{ellsyst}-\eqref{nh-BCs}, 
by means of an auxiliary system of Hammerstein integral equations
\begin{gather}
\begin{aligned}\label{nh-Ham}
u(t)=\int_{0}^{1}k(t,s)\hat{f}_1(s,u(s),v(s),A_1, A_2)\,ds, \\
v(t)= \int_{0}^{1}k(t,s)\hat{f}_2(s,u(s),v(s),A_1, A_2)\,ds,%
\end{aligned}
\end{gather}
where 
\begin{equation*}
k(t,s)=
 \begin{cases} s(1-t), & s\le t, \\
   t(1-s),& s>t.
\end{cases}
\end{equation*}
The integral equations in~\eqref{nh-Ham} share the same \emph{non-negative} kernel and the non-homogeneous terms that occur in  \eqref{nh-BCs} are 
incorporated in the nonlinearities $\hat{f}_1,\hat{f}_2$ (a similar idea has been fruitfully employed in ~\cite{dolo3} also in the context of exterior domains). The existence of positive solutions of \eqref{nh-Ham} is obtained via the well-known Krasnosel'ski\u\i{}-Guo Theorem on cone compressions and cone expansions (see~\cite{guolak}). The Krasnosel'ski\u\i{}-Guo Theorem and, more in general, \emph{topological} methods have been used to study the existence of positive solutions for elliptic equations subject to \emph{homogeneous} BCs on annular domains, see for example the papers by Dunninger and Wang~\cite{dunn-wang1, dunn-wang2}, Lan and Lin~\cite{lan-lin-na},  Lan and Webb~\cite{lanwebb}, Ma~\cite{rma-el}, Wang~\cite{wangh} and references therein.

The study of \emph{nonlocal} BCs, in the framework of ODEs, has been initiated by 1908 by Picone~\cite{Picone}, who considered multi-point BCs. This topic has been developed by a large number of authors. The motivation for this type of study is driven also by the fact that nonlocal problems occur when modelling several phenomena in engineering, physics and life sciences. For an introduction to nonlocal problems we refer to the reviews by Whyburn~\cite{Whyburn}, Conti~\cite{Conti}, Ma~\cite{rma}, Ntouyas~\cite{sotiris} and \v{S}tikonas~\cite{Stik}. 

Nonlocal BCs have been studied also in the the context of elliptic problems, we mention here the papers by Amster and Maurette~\cite{ammau}, Beals~\cite{beals}, Bitsadze and Samarski\u\i {}~\cite{bisamarADAN69}, Browder~\cite{browder}, Schechter~\cite{schechter}, Skubachevski{\u\i}~\cite{skuba1, skuba2}, Wang~\cite{wangy}, Ye and Ke~\cite{yeKe}. In~\cite{webb} Webb considered the existence of \emph{positive} radial solutions for the boundary value problem (BVP)
\begin{gather}
\begin{aligned}\label{jwnel}
\triangle u+ h(|x|)f(u) = 0, \ |x|\in [R_{1},R_{0}], \\
u|_{\partial B_{R_0}}=0\ \text{and}\ (u(R_1\cdot)-\alpha u(R_\eta \cdot))|_{\partial B_{1}}=0,%
\end{aligned}
\end{gather}
where $\alpha>0$ and $R_\eta \in (R_1,R_0)$.

Here we develop a theory for the existence of \emph{nonzero} solutions of systems of Hammerstein integral equations of the type 
\begin{gather}
\begin{aligned}\label{syst-intro}
u(t)=\int_{0}^{1}k_1(t,s)g_1(s)f_1(s,u(s),v(s))\,ds, \\
v(t)= \int_{0}^{1}k_2(t,s)g_2(s)f_2(s,u(s),v(s))\,ds,%
\end{aligned}
\end{gather}
that is well-suited to prove the existence of \emph{nontrivial} radial solutions for a \emph{class} of elliptic systems subject to nonlocal BCs, similar to the ones that occur in \eqref{jwnel}. With this approach the kernels, allowed to \emph{change sign}, take into account the 
nonlocalities in the BCs.

The existence of  \emph{positive} solutions of systems of integral equations of the type \eqref{syst-intro} has been widely studied, see for example \cite{ag-or-pw, chzh1, chzh2, dunn-wang1, dunn-wang2, Goodrich1, Goodrich2,
 jh-rl1, jh-rl2,lan, lan-lin-lms, lan-lin-na, karejde, ya2, yang-zhang} and references therein. Nonzero solutions of systems of Hammerstein integral equations were considered in \cite{df-gi-do}; here we improve the results of  \cite{df-gi-do} in several directions: we allow different growths in the nonlinearities, discuss non-existence results and provide some criteria that involve the spectral radii of some suitable associated linear operators.

We illustrate our theory in the \emph{special case} of a system of nonlinear elliptic BVPs with nonlocal BCs, that generates two \emph{different} kernels in the associated system of integral equations, namely
\begin{gather*}
\begin{aligned}\label{ellbvp}
\Delta u + h_1(|x|) f_1(u,v)=0,\ &|x|\in [R_1,R_0], \\
\Delta v + h_2(|x|) f_2(u,v)=0,\ &|x|\in [R_1,R_0],\\
\frac{\partial u}{\partial r}\bigr\rvert_{\partial B_{R_0}}=0\ \text{and}\
(u(R_1\cdot)-\alpha_1 & u(R_\eta \cdot))|_{\partial B_{1}}=0,\\
\frac{\partial v}{\partial r}\bigr\rvert_{\partial B_{R_0}}=0 \ \text{and}\ 
\bigl(v(R_1\cdot)-\alpha_2 & \frac{\partial v}{\partial r}(R_\xi \cdot)\bigr)|_{\partial B_{1}}=0,%
\end{aligned}
\end{gather*}
where $x\in \mathbb{R}^n $, $\alpha_1, \alpha_2 \in \mathbb{R}$,
$0<R_1<R_0<\infty$, $R_\eta, R_\xi \in (R_1,R_0)$ and $\dfrac{\partial }{\partial r}$ denotes differentiation in the radial direction $r=|x|$.

Here we focus the attention on the existence of solutions that are allowed to \emph{change sign}, in the spirit of the earlier works~\cite{gijwjmaa, gijwjiea}. The approach that we use is topological, relies on classical fixed point index theory and we make use of ideas from the papers~\cite{df-gi-do, gipp-nonlin, gi-pp-ft, gijwjiea, lan-lin-na, lanwebb, webb, jw-tmna, jwkleig}. In the last Section we present an example that illustrates the applicability of our results.
\section{The system of  integral equations}\label{sec2} 
We begin by stating some assumptions on the terms that occur in the system of  Hammerstein integral equations
\begin{gather}
\begin{aligned}\label{syst}
u(t)=\int_{0}^{1}k_1(t,s)g_1(s)f_1(s,u(s),v(s))\,ds, \\
v(t)= \int_{0}^{1}k_2(t,s)g_2(s)f_2(s,u(s),v(s))\,ds,%
\end{aligned}
\end{gather}
namely:
\begin{itemize}
\item For every $i=1,2$, $f_i: [0,1]\times (-\infty,\infty)\times (-\infty,\infty) \to
[0,\infty)$ satisfies Carath\'{e}odory conditions, that is, $f_i(\cdot,u,v)$
is measurable for each fixed $(u,v)$ and $f_i(t,\cdot,\cdot)$ is continuous
for almost every (a.e.) $t\in [0,1]$, and for each $r>0$ there exists $\phi_{i,r} \in
L^{\infty}[0,1]$ such that{}
\begin{equation*}
f_i(t,u,v)\le \phi_{i,r}(t) \;\text{ for } \; u,v\in [-r,r]\;\text{ and
a.\,e.} \; t\in [0,1].
\end{equation*}%
{}
\item For every $i=1,2$, $k_i:[0,1]\times [0,1]\to (-\infty,\infty)$ is
measurable, and for every $\tau\in [0,1]$ we have
\begin{equation*}
\lim_{t \to \tau} |k_i(t,s)-k_i(\tau,s)|=0 \;\text{ for a.\,e.}\, s \in [0,1].
\end{equation*}%
{}
\item For every $i=1,2$, there exist a subinterval $[a_i,b_i] \subseteq
[0,1] $, a function $\Phi_i \in L^{\infty}[0,1]$, and a constant $c_{i} \in
(0,1]$, such that
\begin{align*}
|k_i(t,s)|\leq \Phi_i(s) \text{ for } &t \in [0,1] \text{ and a.\,e.}\, s\in [0,1], \\
k_i(t,s) \geq c_{i}\Phi_i(s) \text{ for } &t\in [a_i,b_i] \text{ and a.\,e.} \, s \in [0,1].
\end{align*}%
{}
\item For every $i=1,2$, $g_i\,\Phi_i \in L^1[0,1]$, $g_i \geq 0$ a.e., and $%
\int_{a_i}^{b_i} \Phi_i(s)g_i(s)\,ds >0$.
{}
\end{itemize}
We work in the space $C[0,1]\times C[0,1]$ endowed with the norm
\begin{equation*}
\| (u,v)\| :=\max \{\| u\| _{\infty },\| v\| _{\infty }\},
\end{equation*}%
where $\| w\| _{\infty}:=\max \{| w(t)|,t\in [0,1] \}$. 

We recall that a \emph{cone} $K$ in a Banach space $X$  is a closed
convex set such that $\lambda \, x\in K$ for $x \in K$ and
$\lambda\geq 0$ and $K\cap (-K)=\{0\}$. Take
\begin{equation*}
\tilde{K_{i}}:=\{w\in C[0,1]:\min_{t\in \lbrack a_{i},b_{i}]}w(t)\geq c_{i}\|
w\| _{\infty }\},
\end{equation*}%
 and consider the cone $K$
in $C[0,1]\times C[0,1]$ defined by
\begin{equation*}
\begin{array}{c}
K:=\{(u,v)\in \tilde{K_{1}}\times \tilde{K_{2}}\}.%
\end{array}%
\end{equation*}
For a \emph{nontrivial} solution of the system \eqref{syst} we mean a solution
$(u,v)\in K$ of \eqref{syst} such that $\|(u,v)\|\neq 0$.
Note that the functions in $\tilde{K_{i}}$ are positive on the
sub-interval $[a_{i},b_{i}]$ but are allowed to change sign in $[0,1]$. This type of cone has been introduced by  Infante and Webb in \cite{gijwjiea}
and is similar to
a cone of \emph{non-negative} functions
 first used by Krasnosel'ski\u\i{}, see e.g. \cite{krzab}, and D.~Guo, see e.g. \cite{guolak}. 

Under our assumptions, we show that the integral operator
\begin{gather}
\begin{aligned}    \label{opT}
T(u,v)(t):=& 
\left(
\begin{array}{c}
 \int_{0}^{1}k_1(t,s)g_1(s)f_1(s,u(s),v(s))\,ds\\
\int_{0}^{1}k_2(t,s)g_2(s)f_2(s,u(s),v(s))\,ds%
\end{array}
\right)
:=
\left(
\begin{array}{c}
T_1(u,v)(t) \\
T_2(u,v)(t)%
\end{array}
\right) ,
\end{aligned}
\end{gather}
leaves the cone $K$ invariant and is compact.

\begin{lem}\label{compact} 
The operator \eqref{opT} maps $K$ into $K$ and is compact.
\end{lem}

\begin{proof}
Take $(u,v)\in K$ such that $\| (u,v)\| \leq r$. Then we have, for $%
t\in \lbrack 0,1]$,
\begin{equation*}
| T_{1}(u,v)(t)| \leq \int_{0}^{1}\Phi
_{1}(s)g_{1}(s)f_{1}(s,u(s),v(s))\,ds
\end{equation*}%
and therefore
\begin{equation*}
\| T_{1}(u,v)\|_\infty \leq \int_{0}^{1}\Phi
_{1}(s)g_{1}(s)f_{1}(s,u(s),v(s))\,ds.
\end{equation*}%
Then we obtain
\begin{eqnarray*}
\min_{t\in \lbrack a_{1},b_{1}]}T_{1}(u,v)(t) &\geq & c_{1}\int_{0}^{1}\Phi _{1}(s)g_{1}(s)f_{1}(s,u(s),v(s))\,ds \\
&\geq & c_{1}\| T_{1}(u,v)\|_\infty  .
\end{eqnarray*}%
Hence we have $T_{1}(u,v)\in \tilde{K_{1}}$. In a similar manner we proceed for $%
T_{2}(u,v)$.\newline
Moreover, the map $T$ is compact since, by routine arguments, the components $T_{i}$ are compact maps.
\end{proof}

The next Lemma summarizes some classical results regarding the fixed point index, for more details see~\cite{Amann-rev, guolak}.
If $\Omega$ is a open bounded subset of a cone $K$ (in the relative
topology) we denote by $\overline{\Omega}$ and $\partial \Omega$
the closure and the boundary relative to $K$. When $\Omega$ is an open
bounded subset of $X$ we write $\Omega_K=\Omega \cap K$, an open subset of
$K$.
\begin{lem} 
Let $\Omega$ be an open bounded set with $0\in \Omega_{K}$ and $\overline{\Omega}_{K}\ne K$. Assume that $F:\overline{\Omega}_{K}\to K$ is
a compact map such that $x\neq Fx$ for all $x\in \partial \Omega_{K}$. Then the fixed point index $i_{K}(F, \Omega_{K})$ has the following properties.
\begin{itemize}
\item[(1)] If there exists $e\in K\setminus \{0\}$ such that $x\neq Fx+\lambda e$ for all $x\in \partial \Omega_K$ and all $\lambda
>0$, then $i_{K}(F, \Omega_{K})=0$.
\item[(2)] If  $\mu x \neq Fx$ for all $x\in \partial \Omega_K$ and for every $\mu \geq 1$, then $i_{K}(F, \Omega_{K})=1$.
\item[(3)] If $i_K(F,\Omega_K)\ne0$, then $F$ has a fixed point in $\Omega_K$.
\item[(4)] Let $\Omega^{1}$ be open in $X$ with $\overline{\Omega^{1}}\subset \Omega_K$. If $i_{K}(F, \Omega_{K})=1$ and
$i_{K}(F, \Omega_{K}^{1})=0$, then $F$ has a fixed point in $\Omega_{K}\setminus \overline{\Omega_{K}^{1}}$. The same result holds if
$i_{K}(F, \Omega_{K})=0$ and $i_{K}(F, \Omega_{K}^{1})=1$.
\end{itemize}
\end{lem}

We use the following (relative) open bounded sets in $K$:
\begin{equation*}
K_{\rho_1,\rho_2} = \{ (u,v) \in K : \|u\|_{\infty}< \rho_1\ \text{and}\ \|v\|_{\infty}< \rho_2\},
\end{equation*}
and
\begin{equation*}
V_{\rho_1,\rho_2} =\{(u,v) \in K: \min_{t\in [a_1,b_1]}u(t)<\rho_1\ \text{and}\ \min_{t\in
[a_2,b_2]}v(t)<\rho_2\}.
\end{equation*}
If $\rho_1=\rho_2=\rho$ we write simply $K_{\rho}$ and $V_{\rho}$. The set $V_\rho$ (in the context of systems) was introduced by the authors in~\cite{gipp-ns} and is equal to the set
called $\Omega^{\rho /c}$ in~\cite{df-gi-do}. $\Omega^{\rho /c}$ is an extension to the case of systems of a set given by Lan~\cite{lan}. 

For our index calculations we make use of the following Lemma, similar to Lemma $5$ of \cite{df-gi-do}. The novelty here is the use of different radii, in the spirit of the paper~\cite{chzh2}. This choice allows more freedom in the growth of the nonlinearities. The proof of the Lemma is similar to the corresponding one in~\cite{df-gi-do} and is omitted.

\begin{lem}  \label{esca} 
  The sets defined above have the following properties:
\begin{itemize}
\item $K_{\rho_1,\rho_2}\subset V_{\rho_1,\rho_2}\subset K_{\rho_1/c_1,\rho_2/c_2}$.
\item $(w_1,w_2) \in \partial V_{\rho_1,\rho_2}$ \; iff \; $(w_1,w_2)\in K$ and $\displaystyle
\min_{t\in [a_i,b_i]} w_i(t)= \rho_i$ for some $i\in \{1,2\}$ and $\displaystyle
\min_{t\in [a_j,b_j]}w_j(t)
\le \rho_j$ for $j\neq i$.
\item If $(w_1,w_2) \in \partial V_{\rho_1,\rho_2}$, then for some $i\in\{1,2\}$ $\rho_i \le w_i(t) \le \rho_i/c_i$
 for each $t \in [a_i,b_i]$ and for $j\neq i$ we have $0 \leq w_j(t) \leq \rho_j/c_j$ for each $t\in [a_j,b_j]$ and $\|w_j\|_\infty \leq \rho_j/c_j$.
\end{itemize}
\end{lem}
\section{Existence results}
We are now able to prove a result concerning the fixed point index on the set $K_{\rho_1,\rho_2}$.

\begin{lem}
\label{ind1b} Assume that

\begin{enumerate}
\item[$(\mathrm{I}_{\rho_1,\rho_2 }^{1})$] \label{EqB} there exist $\rho_1,\rho_2 >0$
such that for every $i=1,2$
\begin{equation}\label{eqmestt}
 f_i^{\rho_1,\rho_2}  < m_i
\end{equation}{}
where
\begin{equation*}
f_{i}^{\rho_1,\rho_2}=\sup \Bigl\{\frac{f_{i}(t,u,v)}{\rho_i }:\;(t,u,v)\in
\lbrack 0,1]\times [ -\rho_1,\rho_1 ]\times [ -\rho_2,\rho_2 ]\Bigr\} \end{equation*}
and
\begin{equation*}
\ \frac{1}{m_{i}}=\sup_{t\in \lbrack 0,1]}\int_{0}^{1}\vert k_{i}(t,s)\vert g_{i}(s)\,ds.
\end{equation*}{}
\end{enumerate}
Then $i_{K}(T,K_{\rho_1,\rho_2})=1$.
\end{lem}

\begin{proof}
We show that $\lambda (u,v)\neq T(u,v)$ for every $(u,v)\in \partial K_{\rho_1,\rho_2 }$
and for every $\lambda \geq 1$; this ensures that the index is 1 on $K_{\rho_1,\rho_2 }$.
In fact, if this does not happen, there exist $\lambda \geq 1$ and $(u,v)\in
\partial K_{\rho_1,\rho_2 }$ such that $\lambda (u,v)=T(u,v)$. Assume, without loss of
generality, that $\| u\| _{\infty }=\rho_1 $ and $\| v\| _{\infty
}\leq \rho_2 $. Then
\begin{equation*}
\lambda u(t)= \int_{0}^{1}k_1(t,s)g_1(s)f_1(s,u(s),v(s))\,ds.
\end{equation*}%
Taking the absolute value we have
\begin{equation*}
\lambda |u(t)|= \Bigl|\int_{0}^{1}k_1(t,s)g_1(s)f_1(s,u(s),v(s))\,ds\Bigr|,
\end{equation*}%
and then the supremum over $[0,1]$ gives
\begin{align*}
\lambda {\rho_1} \leq & \sup_{t\in \lbrack 0,1]}\int_{0}^{1}|k_1(t,s)|g_1(s)f_1(s,u(s),v(s))\,ds \\
&\leq {\rho_1} f_1^{\rho_1,\rho_2}\sup_{t\in \lbrack 0,1]}\int_{0}^{1}|k_1(t,s)|g_1(s)\,ds ={\rho_1} f_1^{\rho_1,\rho_2} \dfrac{1}{m_1}.
\end{align*}
Using the hypothesis \eqref{eqmestt} we obtain $\lambda \rho_1 <\rho_1 .$ This
contradicts the fact that $\lambda \geq 1$ and proves the result.
\end{proof}
\begin{rem}
Take $\omega\in L^1([0,1]\times [0,1])$ and denote by
$$
\omega^+(t,s)=\max\{\omega(t,s),0\},\,\,\, \omega^-(s)=\max\{-\omega(t,s),0\}. 
$$
Then we have
$$
\Bigl|\int_0^1 \omega(t,s)ds \Bigr|\le\max\Bigl\{\int_0^1 \omega^+(t,s)ds,\int_0^1 \omega^-(t,s)ds \Bigr\}\le \int_0^1 |\omega(t,s)|ds,
$$
since $\omega=\omega^+ - \omega^-$ and  $|\omega|=\omega^++\omega^-$.

Using the inequality above, it is possible to relax the growth assumptions on the nonlinearities $f_i$. This is done by replacing the quantity $\dfrac{1}{m_{i}}$ with
\begin{equation*}\label{eqmestt2}
\sup_{t \in [0,1]  }\Bigl\{ \max\Bigl\{\int_0^1k_i^+(t,s)g_i(s)\,ds,\int_0^1k_i^-(t,s)g_i(s)\,ds\Bigr\} \Bigr\};
\end{equation*}{}
this idea has been used, in the case of one equation, in \cite{gi-pp-ft}. 
\end{rem}

We give a first Lemma that shows that the index is 0 on a set
$V_{\rho_1,\rho_2 }$.  

\begin{lem}\label{idx0b1}
Assume that
\begin{enumerate}
\item[$(\mathrm{I}^{0}_{\rho_1,\rho_2})$]  there exist $\rho_1,\rho_2>0$ such that for every $i=1,2$
\begin{equation}\label{eqMest}
f_{i,(\rho_1, \rho_2)} > M_i,
\end{equation}{}
where
\begin{align*}
 f_{1,({\rho_1,\rho_2 })}=& \inf \Bigl\{ \frac{f_1(t,u,v)}{ \rho_1}:\; (t,u,v)\in [a_1,b_1]\times[\rho_1,\rho_1/c_1]\times[-\rho_2/c_2, \rho_2/c_2]\Bigr\},\\
 f_{2,({\rho_1,\rho_2 })}=& \inf \Bigl\{ \frac{f_2(t,u,v)}{ \rho_2}:\; (t,u,v)\in [a_2,b_2]\times[-\rho_1/c_1,\rho_1/c_1]\times[\rho_2, \rho_2/c_2]\Bigr\},\\
    \frac{1}{M_i}=& \inf_{t\in
[a_i,b_i]}\int_{a_i}^{b_i} k_i(t,s) g_i(s)\,ds.
\end{align*}
\end{enumerate}
Then $i_{K}(T,V_{\rho_1,\rho_2})=0$.
\end{lem}

\begin{proof}
Let $e(t)\equiv 1$ for $t\in [0,1]$. Then $(e,e)\in K$. We prove that
\begin{equation*}
(u,v)\ne T(u,v)+\lambda (e,e)\quad\text{for } (u,v)\in \partial
V_{\rho_1,\rho_2 }\quad\text{and } \lambda \geq 0.
\end{equation*}
In fact, if this does not happen, there exist $(u,v)\in \partial V_{\rho_1,\rho_2 }$ and
$\lambda \geq 0$ such that $(u,v)=T(u,v)+\lambda (e,e)$.
Without loss of generality, we can assume that for all $t\in [a_1,b_1]$ we have
$$
\rho_1\leq u(t)\leq {\rho_1/c_1},\ \min u(t)=\rho_1 \ \text{and}\ -\rho_2/c_2\leq v(t)\leq {\rho_2/c_2}.
$$
Then, for $t\in [a_1,b_1]$, we obtain
$$
u(t)= \int_{0}^{1}k_1(t,s)g_1(s)f_1(s,u(s),v(s))\,ds + \lambda e(t),
$$
and therefore

\begin{align*}
  u(t)\geq & \int_{a_1}^{b_1}k_1(t,s)g_1(s)f_1(s,u(s),v(s))\,ds+{\lambda}.
\end{align*}
Taking the minimum over $[a_{1},b_{1}]$ gives
\begin{align*}
\rho_1= \min_{t\in \lbrack a_{1},b_{1}]}u(t)\geq &{\rho_1 }f_{1,(\rho_1 ,\rho_2)}\frac{1}{M_{1}}+{\lambda}.
\end{align*}
Using the hypothesis \eqref{eqMest} we obtain $\rho_1>\rho_1 +\lambda $, a contradiction.
\end{proof}
In the following Lemma we exploit an idea that was used in \cite{gipp-nonlin} and we provide a result of index 0 on $V_{\rho_1,\rho_2 }$ of a different flavour;  here we  control the growth of just one nonlinearity $f_i$, at the cost of
having to deal with a larger domain. Nonlinearities with different growths were considered, with different approaches, in~\cite{chzh1, precup1, precup2, ya1} .
\begin{lem}\label{idx0b3}
Assume that
\begin{enumerate}
\item[$(\mathrm{I}^{0}_{\rho_1,\rho_2})^{\star}$] there exist $\rho_1,\rho_2>0$ such that for some $i\in\{1,2\}$ we have
\begin{equation}\label{eqMest1}
f^*_{i,(\rho_1, \rho_2)}>M_i,
\end{equation}{}
\end{enumerate}
where
\begin{equation*}
f^*_{1,(\rho_1,{\rho_2})}=\inf \Bigl\{ \frac{f_1(t,u,v)}{ \rho_1}:\; (t,u,v)\in [a_1,b_1]\times[0,\rho_1/c_1]\times[-\rho_2/c_2, \rho_2/c_2]\Bigr\}.
\end{equation*}
\begin{equation*}
f^*_{2,(\rho_1,{\rho_2})}=\inf \Bigl\{ \frac{f_2(t,u,v)}{ \rho_2}:\; (t,u,v)\in [a_2,b_2]\times[-\rho_1/c_1,\rho_1/c_1]\times[0, \rho_2/c_2]\Bigr\}.
\end{equation*}
Then $i_{K}(T,V_{\rho_1,\rho_2})=0$.
\end{lem}

\begin{proof}
Suppose that the condition \eqref{eqMest1} holds for $i=1$.
Let  $(u,v)\in \partial V_{\rho_1,\rho_2 }$ and
$\lambda \geq 0$ such that $(u,v)=T(u,v)+\lambda (e,e)$.
So  for all $t\in [a_1,b_1]$ we have $\min u(t)\leq \rho_1$, $0 \leq u(t)\leq \rho_1/c_1$ and $-\rho_2/c_2 \leq v(t)\leq \rho_2/c_2$ and for $t\in [a_2,b_2]$, $\min v(t)\leq \rho_2$.
For $t\in [a_1,b_1]$, as in the proof of Lemma \ref {idx0b1}, we have
\begin{equation*}
  u(t)\geq \int_{a_1}^{b_1}k_1(t,s)g_1(s)f_1(s,u(s),v(s))\,ds+{\lambda}.
\end{equation*}
Taking the minimum over $[a_{1},b_{1}]$ gives
\begin{equation*}
 \min_{t\in [a_{1},b_{1}]}u(t)\geq \rho_1 f_{1,(\rho_1,{\rho_2})}^{\ast}\frac{1}{M_{1}}+{\lambda}.
\end{equation*}
Using the hypothesis \eqref{eqMest1} we obtain $\rho_1 >\rho_1+\lambda$, a contradiction.
\end{proof}
We now state  a result regarding the existence of at least one, two or three nontrivial solutions. The proof follows by the properties of fixed point index and is omitted.
Note that, by expanding the lists in conditions $(S_{5}),(S_{6})$, it is possible to state results for four or more nontrivial solutions, see for 
example the paper~\cite{kljdeds}.

\begin{thm}\label{mult-sys}
The system \eqref{syst} has at least one nontrivial solution
in $K$ if one of the following conditions holds.

\begin{enumerate}

\item[$(S_{1})$]  For $i=1,2$ there exist $\rho _{i},r _{i}\in (0,\infty )$ with $\rho
_{i}/c_i<r _{i}$ such that $(\mathrm{I}_{\rho _{1},\rho_2}^{0})\;\;[\text{or}\;(%
\mathrm{I}_{\rho _{1},\rho_2}^{0})^{\star }]$, $(\mathrm{I}_{r _{1},r_2}^{1})$ hold.

\item[$(S_{2})$] For $i=1,2$ there exist $\rho _{i},r _{i}\in (0,\infty )$ with $\rho
_{i}<r _{i}$ such that $(\mathrm{I}_{\rho _{1},\rho_2}^{1}),\;\;(\mathrm{I}%
_{r _{1},r_2}^{0})$ hold.
\end{enumerate}

The system \eqref{syst} has at least two nontrivial solutions in $K$ if one of the following conditions holds.

\begin{enumerate}

\item[$(S_{3})$] For $i=1,2$ there exist $\rho _{i},r _{i},s_i\in (0,\infty )$
with $\rho _{i}/c_i<r_i <s _{i}$ such that $(\mathrm{I}_{\rho
_{1},\rho_2}^{0})$, $[\text{or}\;(\mathrm{I}_{\rho _{1},\rho_2}^{0})^{\star }],\;\;(%
\mathrm{I}_{r _{1},r_2}^{1})$ $\text{and}\;\;(\mathrm{I}_{s _{1},s_2}^{0})$
hold.

\item[$(S_{4})$] For $i=1,2$ there exist $\rho _{i},r _{i},s_i\in (0,\infty )$
with $\rho _{i}<r _{i}$ and $r _{i}/c_i<s _{i}$ such that $(\mathrm{I}%
_{\rho _{1},\rho_2}^{1}),\;\;(\mathrm{I}_{r _{1},r_2}^{0})$ $\text{and}\;\;(\mathrm{I%
}_{s _{1},s_2}^{1})$ hold.
\end{enumerate}

The system \eqref{syst} has at least three nontrivial solutions in $K$ if one
of the following conditions holds.

\begin{enumerate}
\item[$(S_{5})$] For $i=1,2$ there exist $\rho _{i},r _{i},s_i,\sigma_i\in
(0,\infty )$ with $\rho _{i}/c_i<r _{i}<s _{i}$ and $s _{i}/c_i<\sigma
_{i}$ such that $(\mathrm{I}_{\rho _{1},\rho_2}^{0})\;\;[\text{or}\;(\mathrm{I}%
_{\rho _{1},\rho_2}^{0})^{\star }],$ $(\mathrm{I}_{r _{1},r_2}^{1}),\;\;(\mathrm{I}%
_{s_1,s_2}^{0})\;\;\text{and}\;\;(\mathrm{I}_{\sigma _{1},\sigma_2}^{1})$ hold.

\item[$(S_{6})$] For $i=1,2$ there exist $\rho _{i},r _{i},s_i,\sigma_i\in
(0,\infty )$ with $\rho _{i}<r _{i}$ and $r _{i}/c_i<s _{i}<\sigma _{i}$
such that $(\mathrm{I}_{\rho _{1},\rho_2}^{1}),\;\;(\mathrm{I}_{r
_{1},r_2}^{0}),\;\;(\mathrm{I}_{s _{1},s_2}^{1})$ $\text{and}\;\;(\mathrm{I}%
_{\sigma _{1},\sigma_2}^{0})$ hold.
\end{enumerate}
\end{thm}
In the case of $[a_1,b_1]=[a_2,b_2]$ we can relax the assumptions on the nonlinearities $f_i$. In the following two Lemmas we provide a modification of the conditions   $(\mathrm{I}_{\rho
_{1},\rho_2}^{0})$ and $(\mathrm{I}_{\rho _{1},\rho_2}^{0})^{\star }$, similar to the one in \cite{df-gi-do}. An analogous of the Theorem \ref{mult-sys} holds in this case, we omit the statement of this result.

\begin{lem}\label{idx0b1eq}
Assume that $[a_1,b_1]=[a_2,b_2]=:[a,b]$ and that
\begin{enumerate}
\item[$(\mathrm{\underline{I}}^{0}_{\rho_1,\rho_2})$]  there exist $\rho_1,\rho_2>0$ such that for every $i=1,2$
\begin{equation}\label{eqMesteq}
\underline{f}_{i,(\rho_1, \rho_2)} > M_i,
\end{equation}{}
where
\begin{multline*}
\underline{f}_{1,(\rho_1,{\rho_2})}=\inf \Bigl\{ \frac{f_1(t,u,v)}{ \rho_1}:\; (t,u,v)\in [a,b]\times[\rho_1,\rho_1/c_1]\times[0, \rho_2/c_2]\Bigr\},\\
\underline{f}_{2,(\rho_1,{\rho_2})}=\inf \Bigl\{ \frac{f_2(t,u,v)}{ \rho_2}:\; (t,u,v)\in [a,b]\times[0,\rho_1/c_1]\times[\rho_2, \rho_2/c_2]\Bigr\}.
\end{multline*}
\end{enumerate}
Then $i_{K}(T,V_{\rho_1,\rho_2})=0$.
\end{lem}

\begin{proof}
As in the proof of Lemma \ref {idx0b1} suppose that there exist $(u,v)\in \partial V_{\rho_1,\rho_2 }$ and
$\lambda \geq 0$ such that $(u,v)=T(u,v)+\lambda (e,e)$.
Without loss of generality, we can assume that for all $t\in [a,b]$ we have
$$
\rho_1\leq u(t)\leq {\rho_1/c_1},\\\ \min u(t)=\rho_1 \\\ \text{and  }\\\ 0\leq v(t)\leq {\rho_2/c_2}.
$$
Then, for $t\in [a,b]$, we obtain
\begin{equation*}
  u(t)\geq \int_{a}^{b}k_1(t,s)g_1(s)f_1(s,u(s),v(s))\,ds+{\lambda}.
\end{equation*}
Taking the minimum over $[a,b]$ gives
\begin{equation*}
\rho_1= \min_{t\in [a,b]}u(t)\geq {\rho_1 }\underline{f}_{1,(\rho_1 ,{\rho_2})}\frac{1}{M_{1}}+{\lambda}.
\end{equation*}
Using the hypothesis \eqref{eqMesteq} we obtain $\rho_1>\rho_1 +\lambda $, a contradiction.
\end{proof}

\begin{lem}\label{idx0b3eq}
Assume that $[a_1,b_1]=[a_2,b_2]=:[a,b]$ and that
\begin{enumerate}
\item[$(\mathrm{\underline{I}}^{0}_{\rho_1,\rho_2})^{\star}$] there exist $\rho_1,\rho_2>0$ such that for some $i\in\{1,2\}$ we have
\begin{equation}\label{eqMest1eq}
\underline{f}^*_{i,(\rho_1, \rho_2)}>M_i,
\end{equation}{}
\end{enumerate}
where
\begin{equation*}
\underline{f}^*_{i,(\rho_1,{\rho_2})}=\inf \Bigl\{ \frac{f_i(t,u,v)}{ \rho_i}:\; (t,u,v)\in [a,b]\times[0,\rho_1/c_1]\times[0, \rho_2/c_2]\Bigr\}.
\end{equation*}
Then $i_{K}(T,V_{\rho_1,\rho_2 })=0$.
\end{lem}

\begin{proof}
Suppose that the condition \eqref{eqMest1eq} holds for $i=1$.
Let  $(u,v)\in \partial V_{\rho_1,\rho_2 }$ and
$\lambda \geq 0$ such that $(u,v)=T(u,v)+\lambda (e,e)$.
So  for all $t\in [a,b]$ we have $\min u(t)\leq \rho_1$, $0 \leq u(t)\leq \rho_1/c_1$,  $0 \leq v(t)\leq \rho_2/c_2$ and  $\min v(t)\leq \rho_2$.
Now, the proof follows as the one of Lemma \ref {idx0b3}.
\end{proof}
\section{Non-existence results}
We now show a non-existence result for problem \eqref{syst}.

\begin{thm} Assume that one of the following conditions holds. 
\begin{enumerate}
\item For  $i=1,2$, 
\begin{equation}\label{cond1}
f_i(t,u_1,u_2)<m_i|u_i|\ \text{for every}\ t\in [0,1] \text{        and        } u_i\neq 0.
\end{equation}
\item For  $i=1,2$, 
\begin{equation}\label{cond2}
f_i(t,u_1,u_2)>M_i u_i\ \text{for every}\ t\in [a_i,b_i] \text{        and        }  u_i>0.
\end{equation}
\item There exists $i\in\{1,2\}$ such that \eqref{cond1} is verified for $f_i$ and for  $j\neq i$  condition \eqref{cond2}  is verified for $f_j$.
\end{enumerate}
Then there is no nontrivial solution of the system \eqref{syst} in $K$. 
\end{thm}
\begin{proof}
$(1)$ Assume, on the contrary, that there exists $(u,v)\in K$ such that $(u,v)=T(u,v)$ and $(u,v)\neq (0,0)$. Let, for example, be $\|u\|_\infty \neq 0$. 
Then, for $t\in [0,1]$,
\begin{align*}|u(t)|  =&\left|\int_0^1k_1(t,s)g_1(s)f_1(s,u(s),v(s))ds\right|\\ \le & \int_0^1|k_1(t,s)|g_1(s)f_1(s,u(s),v(s))\,ds\\
< & m_1 \int_0^1|k_1(t,s)|g_1(s)|u(s)|\,ds\\
\le & m_1 \|u\|_\infty\int_0^1|k_1(t,s)|g_1(s)\,ds.
\end{align*}
 Taking the supremum for $t\in [0,1]$, we have
$$
\|u\|_\infty< m_1 \|u\|_\infty\sup_{t\in[0,1]}\int_0^1|k_1(t,s)|g_1(s)\,ds =\|u\|_\infty,
$$
a contradiction.\par
$(2)$ Assume, on the contrary, that there exists $(u,v)\in K$ such that $(u,v)=T(u,v)$ and$(u,v)\neq (0,0)$. Let, for example, be $\|u\|_\infty\neq 0$.  Then, for $t\in [a_1,b_1]$
\begin{align*}
u(t)= & \int_0^1 k_1(t,s)g_1(s)f_1(s,u(s),v(s))ds\geq \int_{a_1}^{b_1} k_1(t,s)g_1(s)f_1(s,u(s),v(s))ds\\ 
> & M_1 \int_{a_1}^{b_1} k_1(t,s)g_1(s)u(s)ds.
\end{align*}
Taking the infimum  for $t\in [a_1,b_1]$, we obtain
$$
\min_{t\in[a_1,b_1]}u(t)> M_1 \inf_{t\in[a_1,b_1]}\int_{a_1}^{b_1} k_1(t,s)g_1(s)u(s)\,ds.
$$
Take $\sigma_1=\min_{t\in[a_1,b_1]}u(t)$. Thus we get
$$
\sigma_1>M_1 \sigma_1 \inf_{t\in[a_1,b_1]}\int_{a_1}^{b_1} k_1(t,s)g_1(s)\,ds =\sigma_1,
$$
a contradiction.\par
$(3)$ Assume, on the contrary, that there exists $(u,v)\in K$ such that $(u,v)=T(u,v)$ and $(u,v)\neq (0,0)$. Let, for example, be $\|u\|_\infty \neq 0$. 
Then the function $f_1$ satisfies either \eqref{cond1} or \eqref{cond2} and the proof follows as in the previous cases.
\end{proof}

\section{Eigenvalue criteria for the existence of nontrivial solutions}\label{seceigen}

In order to state our eigenvalue comparison results, we consider, in a similar way as in \cite{gi-pp-ft}, the following operators on $C[0,1]\times C[0,1]$
\begin{equation*}\label{opL}
L(u,v)(t):=
\left(
\begin{array}{c}
 \int_{0}^{1}|k_1(t,s)|g_1(s)u(s)\,ds\\
\int_{0}^{1}|k_2(t,s)|g_2(s)v(s)\,ds%
\end{array}
\right)
:=
\left(
\begin{array}{c}
L_1(u)(t) \\
L_2(v)(t)%
\end{array}
\right) ,
\end{equation*}
and
\begin{equation*}\label{opL^+}
L^+(u,v)(t):= 
\left(
\begin{array}{c}
 \int_{a_1}^{b_1}k_1^+(t,s)g_1(s)u(s)\,ds\\
\int_{a_2}^{b_2}k_2^+(t,s)g_2(s)v(s)\,ds%
\end{array}
\right)
:=
\left(
\begin{array}{c}
L_1^+(u)(t) \\
L_2^+(v)(t)%
\end{array}
\right).
\end{equation*}
 We denote by $P$ the cone of positive functions, namely 
$$
P:=\{w\in C[0,1]:\ w(t)\geq 0,  t\in [0,1]\}.
$$
\begin{thm}\label{lcomp} 
The operators $L$ and $L^+$ are compact and map $P\times P$ into $(P\times P)\cap K$.
\end{thm}

\begin{proof}
 Note that the operators $L$ and $L^+$ map $P\times P$ into $P\times P$ (because they have a non-negative integral kernel) and are compact. We now show that they map $P\times P$ into $(P\times P)\cap K$. Firstly, we do this for the operator $L$. 
 
We observe that for every $i=1,2$ and for $t\in [0,1]$
\begin{equation*}
|k_i(t,s)|\le \Phi_i(s),
 \end{equation*}
and  that, for $t\in [a_i,b_i],$
 \begin{equation*}
 |k_i(t,s)|=k_i(t,s)\ge c_i\Phi_i(t).
 \end{equation*}
 Thus, with a similar proof as the one in Lemma \ref {compact}, we obtain, for $(u,v)\in P\times P$ and $t\in [0,1]$, $L(u,v)\in K$. 
 A similar proof works for $L^+$, since for every $i=1,2$ and $t\in [0,1]$, we have
 \begin{equation*}
|k_i^+(t,s)|\leq |k_i(t,s)|\le \Phi_i(s),
 \end{equation*}
and, for $t\in [a_i,b_i],$
 \begin{equation*}
 k_i^+(t,s)=k_i(t,s)\ge c_i\Phi_i(t).
 \end{equation*}{}
\end{proof}
We recall that $\lambda$ is an \textit{eigenvalue} of a linear operator $\Gamma$ with corresponding eigenfunction $\varphi$ if $\varphi \neq 0$ and ${\lambda}\varphi=\Gamma \varphi$. The reciprocals of nonzero
eigenvalues are called \emph{characteristic values} of $\Gamma$. 
We will denote the \textit{spectral radius} of $\Gamma$ by $r(\Gamma):=\lim_{n\to\infty}\|\Gamma^n\|^{\frac{1}{n}}$ and its \textit{principal characteristic value} (the reciprocal of the spectral radius) by $\mu(\Gamma)=1/r(\Gamma)$.\par

The following Theorem is analogous to the ones in \cite{jw-gi-jlms,jwkleig} and is proven by using the facts that the considered operators  leave $P\times P$ invariant, that $P\times P$ is reproducing, combined with the well-known Krein-Rutman Theorem. 
\begin{thm}\label{specrad}
For $i=1,2$, the spectral radius of $L_i$ is nonzero and is an eigenvalue of $L_i$ with an eigenfunction in $P$. A similar result holds for $L_i^+$.
\end{thm}

\begin{rem}
As a consequence of the two previous theorems, we have that the above mentioned eigenfunction is in $P\cap \tilde{K}_i$.
\end{rem}
We consider the following operator on $C[a_1,b_1]\times C[a_2,b_2]$:
\begin{equation*}\label{opL^+ri}
\bar{L} ^+(u,v)(t):=
\left(
\begin{array}{c}
 \int_{a_1}^{b_1}k_1^+(t,s)g_1(s)u(s)\,ds\\
\int_{a_2}^{b_2}k_2^+(t,s)g_2(s)v(s)\,ds%
\end{array}
\right)
:=
\left(
\begin{array}{c}
\bar{L} _1^+(u)(t) \\
\bar{L} _2^+(v)(t)%
\end{array}
\right).
\end{equation*}

In the recent papers \cite{jw-lms, jw-tmna}, Webb developed an elegant theory valid for $u_0$-positive linear operators. It turns out that our operators  $\bar{L}^+_i$ fit within this setting and, in particular, satisfy the assumptions of Theorem $3.4$ of \cite{jw-tmna}.  We state here a special case of Theorem $3.4$ of \cite{jw-tmna} that can be used for $\bar{L}^+_i$.
 
\begin{thm}\label{thmjeff}
Suppose that there exist $w\in C[a_i,b_i]\setminus \{0\}$, $w\geq 0$ and $\lambda>0$ such that $$\lambda w(t)\geq \bar{L}^+_iw(t),\ \text{for}\ t\in [a_i,b_i].$$ Then we have $r(\bar{L}^+_i)\leq \lambda$.
\end{thm}

\begin{thm}
\label{idx0aut1}
Assume that
\begin{enumerate}
\item[$(\mathrm{I}^{0}_{0^+})$]  there exist $\varepsilon>0$ and $\rho_0>0$ such that one of the following conditions holds:
\end{enumerate}
\begin{equation}\label{eqmu+}
f_{1}(t,u,v) \geq (\mu(L^+_1)+\varepsilon)u,\,\,\text{for }  (t,u,v)\in [a_1,b_1]\times [0,\rho_0]\times [-\rho_0,\rho_0];
\end{equation}
\begin{equation*}\label{eqmu+2}
 f_{2}(t,u,v) \geq(\mu(L^+_2)+\varepsilon)v,\,\,\text{for  }  (t,u,v)\in [a_2,b_2]\times [-\rho_0,\rho_0]\times [0,\rho_0].
\end{equation*}
Then $i_{K}(T,K_{\rho})=0$ for each $\rho\in (0,\rho_{0}]$.
\end{thm}

\begin{proof}
 Let $\rho\in(0,\rho_0]$. We show that
$(u,v)\ne T(u,v)+\lambda(\varphi_1,\varphi_2)$ for all $(u,v)$ in $\partial K_\rho$ and $\lambda\geq 0$,
where $\varphi_i\in \tilde{K}_i\cap P$ is the eigenfunction of $L^+_i$ with $\|\varphi_i\|_{\infty}=1$ corresponding to the eigenvalue $1/\mu(L^+_i)$. This implies that $ i_{K}(T,K_{\rho})=0$.\\
Assume, on the contrary, that there exist $(u,v)\in\partial K_\rho$ and $\lambda\geq0$ such that $(u,v)=T(u,v)+\lambda(\varphi_1,\varphi_2)$.\\ We distinguish two cases. Firstly we discuss the case $\lambda>0$. Suppose that \eqref{eqmu+} holds.
This implies that, for $t\in [a_1,b_1],$ we have
\begin{align*}
 u(t)=& \int_0^1k_1(t,s)g_1(s)f_1(s,u(s),v(s))ds +\lambda \varphi_1(t)\\ \geq & \int_{a_{1}}^{b_{1}} k^+_{1}(t,s)g_1(s)f_1(s,u(s),v(s))ds+\lambda \varphi_1(t) \\ \geq &  (\mu( L^+_1)+\varepsilon) \int_{a_{1}}^{b_{1}} k^+_1(t,s)g_1(s)u(s)ds+\lambda \varphi_1(t)\\ 
 >&\mu( L^+_1) \int_{a_{1}}^{b_{1}} k^+_1(t,s)g_1(s)u(s)ds+\lambda \varphi_1(t) \\=&\mu( L^+_1)  L^+_1u(t)+\lambda\varphi_1(t).
 \end{align*}
Moreover, we have  $u(t)\ge\lambda\varphi_1(t)$  and then $ L^+_1u(t)\ge\lambda L^+_1\varphi_1(t)\ge \dfrac{\lambda}{\mu( L^+_1)}\varphi_1(t)$ in such a way that we obtain 
$$
u(t)\ge\mu( L^+_1) L^+_1u(t)+\lambda\varphi_1(t)\ge2\lambda\varphi_1(t),\ \text{   for   } t\in[a_1,b_1].
$$
By iteration, we deduce that, for $ t\in[a_1,b_1]$, we get
$$
u(t)\ge n\lambda\varphi_1(t)  \text{   for every  } n\in\mathbb {N},
$$
 a contradiction because $\|u\|_{\infty}\leq\rho$.\\
  Now we consider the case $\lambda=0$.  We have, for $t\in [a_1,b_1]$, 
\begin{align*}
 u(t)=&\int_0^1k_1(t,s)g_1(s)f_1(s,u(s),v(s))ds\\
  \geq&\int_{a_1}^{b_1} k^+_1(t,s)g_1(s)f_1(s,u(s),v(s))ds \geq(\mu( L^+_1)+\varepsilon)  L^+_1u(t).
\end{align*}
Since $ L^+_1\varphi_1(t)=r( L^+_1)\varphi_1(t)$ for $t\in[0,1]$, we have, for $t\in[a_1,b_1]$,
$$
\bar{L}^+_1 \varphi_1(t)= L^+_1\varphi_1(t)=r(L^+_1)\varphi_1(t),
$$
and we obtain $r(\bar{L}^+_1)\geq r(L^+_1)$. On the other hand, we have, for $t\in [a_1,b_1]$, 
\begin{equation*}
 u(t)\geq(\mu(L^+_1)+\varepsilon)  L^+_1u(t)=(\mu(L^+_1)+\varepsilon) \bar{L}^+_1 u(t).
\end{equation*}
where $u(t)>0$. Thus, utilizing Theorem~\ref{thmjeff}, we have  $r(\bar{L}^+_1)\leq \dfrac{1}{\mu(L^+_1)+\varepsilon}$
and therefore $r({L}^+_1)\leq \dfrac{1}{\mu( L^+_1)+\varepsilon}$
and thus $\mu( L^+_1)+\varepsilon\leq \mu(L^+_1)$, a contradiction. \par
 \end{proof}
 
 \begin{rem}\label{idx0aut2}
 Note that condition \eqref{eqmu+} holds, for example, if 
 $$
 \mu( L^+_1)<\liminf_ {u\to 0^+} \inf\limits_{t \in [a_1,b_1]} \frac{f_1(t,u,v)}{u},\  \text{uniformly w.r.t.}\, v\in \mathbb{R}.
 $$
A similar type of condition has been used in \cite{chzh1}.
\end{rem} 
 
  \begin{thm} \label{idx0aut3}
Assume that
\begin{enumerate}
\item[$(\mathrm{I}^{0}_{\infty})$]  there exists $R_1>0$ such that  the following conditions hold:
\end{enumerate}
\begin{equation}\label{eqmu+in}
f_{1}(t,u,v) \geq (\mu(L^+_1)+\varepsilon)u,\,\,\text{for }  (t,u,v)\in [a_1,b_1]\times [cR_1,+\infty)\times \mathbb{R};
\end{equation}
\begin{equation*}\label{eqmu+2in}
 f_{2}(t,u,v) \geq (\mu(L^+_2)+\varepsilon)v,\,\,\text{for }  (t,u,v)\in [a_2,b_2]\times \mathbb{R}\times [cR_1,+\infty).
\end{equation*}
Then $i_{K}(T,K_{R})=0$ for each $R\geq R_1$.
\end{thm}

\begin{proof}
 Let $R\geq R_1$.  We show that
$(u,v)\ne T(u,v)+\lambda(\varphi_1,\varphi_2)$ for all $(u,v)$ in $\partial K_R$ and $\lambda\geq 0$,
where $\varphi_i\in \tilde{K}_i\cap P$ is the eigenfunction of $L^+_i$ with $\|\varphi_i\|_{\infty}=1$ corresponding to the eigenvalue $1/\mu(L^+_i)$. This implies that $ i_{K}(T,K_{R})=0$.\\
Assume, on the contrary, that there exist $(u,v)\in\partial K_R$ and $\lambda\geq0$ such that $(u,v)=T(u,v)+\lambda(\varphi_1,\varphi_2)$.\\ Suppose that $\|u\|_{\infty}=R$ and $\|v\|_{\infty}\leq R$. 
 We have $u(t)\ge c\|u\|_{\infty}=c R\ge c R_1$ for $t\in[a_1,b_1]$, thus condition \eqref{eqmu+in} holds. Hence, we have $f(t,u(t),v(t))\geq(\mu( L^+_1)+\varepsilon)u(t)$ for $t\in[a_1,b_1]$. This implies, proceeding as in the proof of Theorem \ref{idx0aut1} for the case $\lambda>0$, that for $t\in [a_1,b_1]$
$$
u(t)\ge\mu(L^+_1)  L^+_1u(t)+\lambda\varphi_1(t)\ge2\lambda\varphi_1(t).
$$
Then $u(t)\ge n\lambda\varphi_1(t)$ for every $n\in\mathbb {N}$, a contradiction because $\|u\|_{\infty}=R$.\\
 The proof in the case $\lambda=0$ is treated as in the proof of Theorem \ref{idx0aut1}.
 \end{proof}

\begin{thm}
\label{idx1aut1}
Assume that
\begin{enumerate}
\item[$(\mathrm{I}^{1}_{0^+})$]  there exist $\varepsilon>0$ and $\rho_0>0$ such that  the following conditions hold:
\end{enumerate}
\begin{equation*}\label{eq1mu+}
f_{1}(t,u,v) \leq (\mu(L_1)-\varepsilon)|u|,\,\,\text{for all }(t,u,v)\in [0,1]\times [-\rho_0,\rho_0] \times[-\rho_0,\rho_0] ;
\end{equation*}
\begin{equation*}\label{eq1mu+2}
 f_{2}(t,u,v) \leq (\mu(L_2)-\varepsilon)|v|,\,\,\text{for all }(t,u,v)\in [0,1]\times [-\rho_0,\rho_0] \times[-\rho_0,\rho_0] .
\end{equation*}
Then $i_{K}(T,K_{\rho})=1$ for each $\rho\in (0,\rho_0]$.
\end{thm}

\begin{proof}
 Let $\rho\in (0,\rho_0]$.  We prove that $T(u,v)\ne\lambda (u,v)$ for $(u,v)\in\partial K_\rho$ and $\lambda\ge 1$, which implies $ i_{K}(T,K_{\rho})=1$. In fact, if we assume otherwise, then there exists $(u,v)\in\partial K_\rho$ and $\lambda\ge1$ such that $\lambda (u,v)=T(u,v)$. Therefore,
\begin{align*}
 |u(t)|\leq&\lambda |u(t)|=  |T_1(u,v)(t)|  =  \left|\int_0^1k_1(t,s)g_1(s)f_1(s,u(s),v(s))ds\right|\\ \le &\int_0^1|k_1(t,s)|g_1(s)f_1(s,u(s),v(s))ds \le (\mu(L_1)-\varepsilon)\int_0^1|k_1(t,s)|g_1(s)|u(s)|ds\\ = &(\mu(L_1)-\varepsilon)L_1 |u|(t).
 \end{align*}
Thus, we have that, for $ t\in [0,1]$,
\begin{align*}
|u(t)|\le & (\mu(L_1)-\varepsilon)L_1[(\mu(L_1)-\varepsilon)L_1|u|(t)]\\ =&(\mu(L_1)-\varepsilon)^2L_1^2|u|(t)\le\cdots\le(\mu(L_1)-\varepsilon)^nL_1^n|u|(t),
 \end{align*}
thus, taking the norms, $1\le(\mu(L_1)-\varepsilon)^n\|L_1^n\|$, and then
$$1\le(\mu(L_1)-\varepsilon)\lim_{n\to\infty}\|L_1^n\|^\frac{1}{n}=\frac{\mu(L_1)-\varepsilon}{\mu(L_1)}<1,$$
a contradiction.
\end{proof}
\begin{thm}
\label{idx1aut2}
Assume that
\begin{enumerate}
\item[$(\mathrm{I}^{1}_{\infty})$]  there exist $\varepsilon>0$ and $R_1>0$ such that  the following conditions hold:
\end{enumerate}
\begin{equation*}\label{eq1muin}
f_{1}(t,u,v) \le (\mu(L_1)-\varepsilon)|u|,\,\text{for } |u|\geq R_1, |v|\geq R_1,\,\, \text{and a.e.  } t\in [0,1];
\end{equation*}
\begin{equation*}\label{eq1muin2}
 f_{2}(t,u,v) \le  (\mu(L_2)-\varepsilon)|v|,\,\,\text{for  } |u|\geq R_1, |v|\geq R_1,\,\, \text{and a.e.  } t\in [0,1].
\end{equation*}
 Then there exists $R_{0}$ such that $i_{K}(T,K_{R})=1$  for each  $R > R_{0}$.
\end{thm}
\begin{proof}

Since the functions $f_i$ satisfy Carath\'{e}odory condition, there exists  $\phi_{i,R_1} \in L^{\infty}[0,1]$ such that{}
\begin{equation*}
f_i(t,u,v)\le \phi_{i,R_1}(t) \;\text{ for } \; u,v\in [-R_1,R_1]\;\text{ and
a.\,e.} \; t\in [0,1].
\end{equation*}
 Hence, we have
\begin{equation}\label{supest}
f_1(t,u,v)\le(\mu(L_1)-\varepsilon)|u| +\phi_{1,R_1}(t)\ \text{for all}\  u,v\in \mathbb {R}\, \text{and a.e. }\ t\in [0,1],
\end{equation}
and
\begin{equation*}\label{supest1}
f_2(t,u,v)\le(\mu(L_2)-\varepsilon)|v| +\phi_{2,R_1}(t)\ \text{for all}\  u,v\in \mathbb {R}\ \text{and a. e. }\ t\in [0,1].
\end{equation*}
Denote by $\Id$  the identity operator. Since for $i=1,2$ the operators $(\mu(L_i)-\varepsilon) L_i$ have spectral radius less than one, we have that  the operators $(\Id-(\mu(L_i)-\varepsilon)L_i)^{-1}$ exist and are bounded. Moreover, from the Neumann series expression, 
$$
(\Id-(\mu(L_i)-\varepsilon) L_i)^{-1}=\sum_{k=0}^\infty((\mu(L_i)-\varepsilon) L_i)^k
$$
we obtain that  $(\Id-(\mu(L_i)-\varepsilon) L_i)^{-1}$ map $P$ into $P$, since the operators $L_i$ have this property.

Take for $i=1,2$ 
$$
C_i:=\int_{0}^{1}\Phi_i(s)g_i(s)\phi_{i,R_1}(s)ds,
$$
and
$$
 R_0:=\max\{\|(\Id-(\mu(L_i)-\varepsilon) L_i)^{-1} C_i \|_{\infty},\,\,i=1,2 \} \in\mathbb {R}.
$$
 Now we prove that for each $R>R_0$, $T(u,v)\ne\lambda (u,v)$ for all $(u,v)\in\partial K_R$ and $\lambda\ge 1$, which implies $ i_{K}(T,K_{R})=1$. Otherwise there exist $(u,v)\in\partial K_R$ and $\lambda\ge 1$ such that $\lambda (u,v)=T(u,v)$. Suppose that $\|u\|_{\infty}=R$ and $\|v\|_{\infty}\leq R$. \\
 From the  inequality \eqref{supest}, we have, for $t\in [0,1]$,
\begin{align*}
|u(t)|\leq\lambda |u(t)|=  |T_1(u,v)(t)|  = & \left|\int_0^1k_1(t,s)g_1(s)f_1(s,u(s),v(s))ds\right|\\
\le &\int_0^1|k_1(t,s)|g_1(s)f_1(s,u(s),v(s))ds\\  
\le & (\mu(L_1)-\varepsilon)\int_0^1|k_1(t,s)|g_1(s)|u(s)|ds+C_i\\ = &(\mu(L_i)-\varepsilon)L_1 |u|(t)+C_1,
\end{align*}
which implies 
$$
 (\Id-(\mu(L_1)-\varepsilon) L_1)|u|(t)\le C_1.
 $$
 Since $(\Id-(\mu(L_1)-\varepsilon) L)^{-1}$ is non-negative, we have
$$
|u(t)|\le (\Id-(\mu(L_1)-\varepsilon) L_1)^{-1} C_1\leq R_0.
$$
Therefore, we have $\|u\|_{\infty}\le R_0<R$, a contradiction.
\end{proof}

The index results in Sections \ref{sec2} and \ref{seceigen} can be combined  in order to establish results on existence of multiple nontrivial solutions for the system \eqref{syst}, we refer to \cite{lan-lin-na} for similar statements.
\section{An auxiliary system of ODEs}\label{odes}
We now present some results regarding the following system of ODEs 
\begin{gather}
\begin{aligned}\label{1syst}
u''(t) + g_1(t) f_1(t,u(t),v(t)) = 0, \quad \text{a.e. on } [0,1], \\
v''(t) + g_2(t) f_2(t,u(t),v(t)) = 0, \quad \text{a.e. on } [0,1],%
\end{aligned}
\end{gather}
with the BCs
\begin{gather}
\begin{aligned}\label{1BC}
u'(0)=0,\ {\alpha}_1u({\eta})=u(1), \; 0 < {\eta}<1, \\
v'(0)=0,\ v(1)=\alpha_2 v'(\xi), \; 0 < {\xi}<1.
\end{aligned}
\end{gather}
Here we focus on the case $\alpha_1<0$, $0<\alpha_2<1-\xi$, that leads to the case of solutions that are \emph{positive} on some sub-intervals of $ [0,1]$ and are allowed to change sign elsewhere.

To the system \eqref{1syst}-\eqref{1BC} we associate the system of Hammerstein integral equations
\begin{gather}
\begin{aligned}\label{syst2}
u(t)=\int_{0}^{1}k_1(t,s)g_1(s)f_1(s,u(s),v(s))\,ds, \\
v(t)= \int_{0}^{1}k_2(t,s)g_2(s)f_2(s,u(s),v(s))\,ds,%
\end{aligned}
\end{gather}
where the Green's functions are given by
\begin{equation} \label{ker1}
k_1(t,s)=\dfrac{1}{1-\alpha_1}(1-s)-\begin{cases}
\dfrac{\alpha_1}{1-\alpha_1}(\eta -s), &  s \le \eta\\ \quad 0,&
s>\eta
\end{cases}
 - \begin{cases} t-s, &s\le t, \\ \quad 0,&s>t,
\end{cases}
\end{equation}
and 
\begin{equation}\label{ker2}
k_2(t,s)=(1-s)-\begin{cases} \alpha_2, &  s \le \xi\\ \quad 0,& s>\xi
\end{cases}
 - \begin{cases} t-s, &s\le t, \\ \quad 0,&s>t.
\end{cases}
\end{equation}
The Green's function $k_1$ has been studied in \cite{gijwjiea}, where it was shown that we may take 
$$\Phi_1(s)= 1-s,$$ arbitrary $[a_1,b_1]\subset[0,\eta]$ and $c_1={(1-\eta)}/{(1-\alpha_1)}$. 

Regarding $k_2$, this has been studied in \cite{giems}; we may take 
$$
\Phi_2(s)=1-s, 
$$ 
arbitrary $[a_2,b_2]\subset [0,\xi]$ and $c_2=1-\alpha_2-\xi$.

The results of the previous Sections, for example Theorem \ref{mult-sys}, can be applied to the system~\eqref{syst2}.
\subsection{Optimal intervals}
We now assume that $g_1 =g_2\equiv 1$ and we seek the  `optimal' $[a_i,b_i]$ such that 
$$M_i(a_i,b_i)=\Bigl(\inf_{t\in
[a_i,b_i]}\int_{a_i}^{b_i} k_i(t,s) \,ds\Bigl)^{-1}$$ is a minimum. 
This type of problem has been tackled in the past in the case of second and higher order BVPs in \cite{Cab1, gipp-cant, gi-pp, paola, jwpomona, jwwcna04, jwgi-lms-II}.

Since in  $[0,1]\times  [0,1]$ the kernel $k_1$ is non-positive only for $$\dfrac{1-\alpha_1 \eta}{1-\alpha_1}\leq t\leq1\ \text{and}\ 0\leq s \leq \dfrac{1-\alpha_1}{-\alpha_1}t+\dfrac{1}{\alpha_1},$$ by direct calculation, we have
\begin{equation*}
\int_0^1 |k_1(t,s)|\,ds =\begin{cases} -t^2/2+\dfrac{1}{1-\alpha_1}( \dfrac{\eta^2}{2}-\alpha_1\eta^2+\dfrac{1}{2})-\eta^2/2=:\vartheta_1(t), &  0 \le t\leq \dfrac{1-\alpha_1 \eta}{1-\alpha_1},\\ 
 \dfrac{-\alpha_1+2}{-2\alpha_1}t^2+\dfrac{2}{\alpha_1} t+\dfrac{-\alpha_1-\alpha_1^2\eta^2+2}{-2\alpha_1(1-\alpha_1)}=:\vartheta_2(t), & \dfrac{1-\alpha_1 \eta}{1-\alpha_1}\leq t\leq 1,
\end{cases}
 \end{equation*}
 and therefore we obtain
 \begin{align*}
1/m_1=& \sup_{t \in [0,1]} \int_0^1 |k_1(t,s)|\,ds \\ 
=&\begin{cases} \dfrac{1}{1-\alpha_1}( \dfrac{\eta^2}{2}-\alpha_1\eta^2+\dfrac{1}{2})-\eta^2/2=\vartheta_1(0), &  \text{if}-2\alpha_1\eta^2+\alpha_1+1\geq 0,\\ 
 \dfrac{-\alpha_1+2}{-2\alpha_1}+\dfrac{2}{\alpha_1} +\dfrac{-\alpha_1-\alpha_1^2\eta^2+2}{-2\alpha_1(1-\alpha_1)}=\vartheta_2(1),& \text{if} -2\alpha_1\eta^2+\alpha_1+1\leq 0.
\end{cases}
\end{align*}
Firstly we note that $\dfrac{1-\alpha_1\eta}{1-\alpha_1}\geq \eta$. For arbitrary $0 \leq a< b\leq \eta$, the kernel $k_1$ is a positive, non-increasing function of $t$.
Thus we have
$$
1/M_1(a,b)=\min_{t\in[a,b]} \int_{a}^{b} k_1(t,s)\,ds=\int_a^b
k_1(b,s)\,ds.
$$
Note that $\inf_{0\leq a <b}M_1(a,b)=M_1(0,b)$ and we get
$$
1/M_1(0,b)=\int_{0}^{b}k_1(b,s)\,ds=\Bigl(\frac{1-\alpha_1\eta}{1-\alpha_1}-b\Bigr)b
$$
Now we have
$$
\max_{0< b \leq \eta} \Bigl\{\Bigl(\frac{1-\alpha_1\eta}{1-\alpha_1}-b\Bigr)b\Bigr\}=\begin{cases}
\dfrac{(1-\alpha_1\eta)^2}{4(1-\alpha_1)^2},& \text{ if }  \dfrac{1-\alpha_1\eta}{2(1-\alpha_1)}<\eta,\\
\dfrac{\eta(1-\eta)}{1-\alpha_1},&  \text{ if }  \dfrac{1-\alpha_1\eta}{2(1-\alpha_1)}\geq \eta.
\end{cases}
$$

Therefore we may take as optimal interval
$$
[a_1,b_1]=\begin{cases}
[0, \frac{1-\alpha_1\eta}{2(1-\alpha_1)}],& \text{ if }  \frac{1-\alpha_1\eta}{2(1-\alpha_1)}<\eta,\\
[0, \eta],&  \text{ if }  \frac{1-\alpha_1\eta}{2(1-\alpha_1)}\geq \eta.
\end{cases}
$$

The kernel $k_2$ in  $[0,1]\times  [0,1$] is non-positive only for $$1-\alpha_2\leq t\leq1\ \text{and}\ 0\leq s \leq \xi;$$ by direct calculation, we have
\begin{equation*}
\int_0^1 |k_2(t,s)|\,ds =\begin{cases} -t^2/2-\alpha_2\xi+1/2=:\theta_1(t), &  0 \le t\leq1-\alpha_2,\\  -t^2/2+2\xi t-2\xi +\alpha_2 \xi+1/2=:\theta_2(t),& 1-\alpha_2\leq t\leq 1,
\end{cases}
 \end{equation*}
 and therefore we obtain
$$
1/m_2= \sup_{t \in [0,1]} \int_0^1 |k_2(t,s)|\,ds =\max\{\theta_1(0), \theta_2(1)\}=\theta_1(0)=-\alpha_2 \xi +1/2.
$$ 
 For arbitrary $0 \leq a< b\leq \xi$, the kernel $k_2$ is a positive, non-increasing function of $t$.
Thus we have
$$
1/M_2(a,b)=\min_{t\in[a,b]} \int_{a}^{b} k_2(t,s)\,ds=\int_a^b
k_2(b,s)\,ds.
$$
Note that $\inf_{0\leq a <b}M_2(a,b)=M_2(0,b)$ and we get
$$
1/M_2(0,b)=\int_{0}^{b}k_2(b,s)\,ds=\Bigl(\frac{1-\alpha_1\eta}{1-\alpha_1}-b\Bigr)b
$$
Now we have
$$
\max_{0< b \leq \xi} \{(1-\alpha_2)b-b^2\}=\begin{cases}
\dfrac{(1-\alpha_2)^2}{4},& \text{ if }  \dfrac{1-\alpha_2}{2}<\xi,\\
(1-\alpha_2)\xi-\xi^2,&  \text{ if }  \dfrac{1-\alpha_2}{2}\geq \xi.
\end{cases}
$$

Therefore we may take as optimal interval
$$
[a_2,b_2]=\begin{cases}
[0, \frac{1-\alpha_2}{2}],& \text{ if }  \frac{1-\alpha_2}{2}<\xi,\\
[0, \xi],&  \text{ if }  \frac{1-\alpha_2}{2}\geq \xi.
\end{cases}
$$

\section{Radial solutions of systems of elliptic PDEs}\label{radsolpdes}
We now turn back our attention to the systems of BVPs
\begin{gather}
\begin{aligned}\label{ellbvp-secapp}
\Delta u + h_1(|x|) f_1(u,v)=0,\ |&x|\in [R_1,R_0], \\
\Delta v + h_2(|x|) f_2(u,v)=0,\ |&x|\in [R_1,R_0],\\
\frac{\partial u}{\partial r}\bigr\rvert_{\partial B_{R_0}}=0\ \text{and}\
(u(R_1\cdot)-\alpha_1  &u(R_\eta \cdot))|_{\partial B_{1}}=0,\\
\frac{\partial v}{\partial r}\bigr\rvert_{\partial B_{R_0}}=0 \ \text{and}\ 
\bigl(v(R_1\cdot)-\alpha_2 & \frac{\partial v}{\partial r}(R_\xi \cdot)\bigr)|_{\partial B_{1}}=0,%
\end{aligned}
\end{gather}
where $x\in \mathbb{R}^n $, $\alpha_1<0$, $0<\alpha_2<1$,
$0<R_1<R_0<\infty$, $R_\eta, R_\xi \in (R_1,R_0)$.

Consider  in $\mathbb{R}^n$, $n\ge 2$, the equation
\begin{equation}\label{eqell}
\triangle w+ h(|x|)f(w) = 0, \  \text{for a.e.}\  |x|\in
[R_{1},R_{0}].
\end{equation}
\par\noindent
with the BCs
\begin{equation*}\label{ellBC1}
\frac{\partial w}{\partial r}\bigr\rvert_{\partial B_{R_0}}=0\ \text{and}\
(w(R_1\cdot)-\alpha_1  w(R_\eta \cdot))|_{\partial B_{1}}=0,
\end{equation*}
or
\begin{equation*}\label{ellBC2}
\frac{\partial w}{\partial r}\bigr\rvert_{\partial B_{R_0}}=0 \ \text{and}\ 
\bigl(w(R_1\cdot)-\alpha_2 \frac{\partial w}{\partial r}(R_\xi \cdot)\bigr)|_{\partial B_{1}}=0. %
\end{equation*}

In order to establish the existence of radial solutions $w=w(r)$, $r=|x|$, we proceed as in \cite{lan, lan-lin-na, lanwebb} and we  rewrite
\eqref{eqell} in the form
\begin{equation}\label{eqinterm}
w''(r) + \dfrac{n-1}{r}w'(r) + h(r)f(w(r))= 0 \quad\text{a.e. on }
[R_{1}, R_{0}].
\end{equation}
Set
$w(t)=w(r(t))$ where, for $n\geq 3$,
$$
r(t)=({\gamma}+({\beta}-{\gamma})t)^{-1/(n-2)},\ \text{for}\ t\in
[0,1],
$$
with ${\gamma}=R_{0}^{-(n-2)}$ and ${\beta}=R_{1}^{-(n-2)}$,  and for $n=2$,
$$r(t)=R_0^{1-t}R_1^{t},\ \text{for}\ t\in
[0,1].
$$
Take for $n\geq 3$
$$
{\phi}(t)=(({\beta}-{\gamma})/(n-2))^{2}({\gamma}
+({\beta}-{\gamma})t)^{-2(n-1)/(n-2)},
$$
and for $n=2$  
$$
\phi(t)=\bigl(R_0(1-t)\log \frac{R_0}{R_1}\bigr)^2.
$$ 
Then the equation ~\eqref{eqinterm}
becomes 
\begin{equation*}\label{eqdef}
w''(t) + {\phi}(t) h(r(t)) f(w(t)) = 0, \  \text{a.e. on}\ [0,1],
\end{equation*}
 subject to the BCs
 \begin{equation*}\label{BC1secapp}
w'(0)=0,\; {\alpha}w({\eta})=w(1), \; 0 < {\eta}<1,
\end{equation*}
 or 
\begin{equation*}\label{BC2secapp}
w'(0)=0,\; {\alpha}w'({\xi})=w(1), \; 0 < {\xi}<1.
 \end{equation*}
Thus, to the system \eqref{ellbvp-secapp} we can associate the system of Hammerstein integral equations
\begin{gather}
\begin{aligned}\label{int-secapp}
u(t)=\int_{0}^{1}k_1(t,s)g_1(s)f_1(u(s),v(s))\,ds, \\
v(t)= \int_{0}^{1}k_2(t,s)g_2(s)f_2(u(s),v(s))\,ds,%
\end{aligned}
\end{gather}
where $k_1$ is as in \eqref{ker1}, $k_2$ is as in \eqref{ker2} and $$g_i(t):={\phi}(t) h_i(r(t)).$$ 

The results of the previous Sections can be applied to the system \eqref{int-secapp}, yielding results for the system \eqref{ellbvp-secapp}, we refer to \cite{lan-lin-na, lanwebb} for the results that may be stated.

We illustrate in the following example that all the constants that occur in the Theorem~\ref{mult-sys} can be computed.
\begin{ex}
Consider  in $\mathbb{R}^2$,  the system of BVPs
\begin{gather}
\begin{aligned}\label{ellbvpex}
\Delta u +  f_1(u,v)=0&,\ |x|\in [1,e], \\
\Delta v + f_2(u,v)=0&,\ |x|\in [1,e],\\
\frac{\partial u}{\partial r}\bigr\rvert_{\partial B_{e}}=0\ \text{and}\
(u(\cdot)+  &u(\sqrt{2} \cdot))|_{\partial B_{1}}=0,\\
\frac{\partial v}{\partial r}\bigr\rvert_{\partial B_{e}}=0 \ \text{and}\ 
\bigl(v(\cdot)-\frac{1}{4} & \frac{\partial v}{\partial r}(\sqrt[4]{e^3} \cdot)\bigr)|_{\partial B_{1}}=0.%
\end{aligned}
\end{gather}
To the system~\eqref{ellbvpex} we associate the system of second order ODEs
\begin{gather*}
\begin{aligned}\label{ellbvpodeex}
u''(t)+ e^2(1-t)^2 f_1(u(t),v(t))=0&,\ t\in [0,1], \\
v''(t) + e^2(1-t)^2 f_2(u(t),v(t))=0&,\ t\in [0,1],\\
u'(0)=0,\; u(1/2)+u(1)=0,\\
v'(0)=0,\; v'(1/4)=4v(1).%
\end{aligned}
\end{gather*}
Now we have
\begin{multline*}
 \frac{1}{m_{1}}=\sup_{t\in [0,1]}\int_{0}^{1}\vert k_{1}(t,s)\vert g_{1}(s)\,ds\\
 =\max \Bigl\{ 
 \sup_{t\in [0,1/2]}
 \bigl\{  -{\frac {e^2}{384}}  \left( -128\,{t}^{3}+32\,{t}^{4}+192\,{t}^{2}-65 \right)  \bigr\},\\
 \sup_{t\in [1/2,3/4]}
  \bigl\{
-{\frac {e^2 }{384}} \left( 160\,{t}^{4}+864\,{t}^{2}-608\,{t}^{3}+19-400\,t \right)\bigr\}, \\
 \sup_{t\in [3/4,1]}   \bigl\{ 
 \frac{1}{2}t+\frac{5}{4}\, e^2{t}^{4}+15/2\,{t}^{2} e^2-5\,{t}^{3} e^2+{\frac {467}{384}}\, e^2-{\frac {119}{24}}\,t e^2- \frac{3}{8}
 \bigr\} \Bigr\},
 \end{multline*}{}
and 
\begin{multline*}
 \frac{1}{m_{2}}=\sup_{t\in [0,1]}\int_{0}^{1}\vert k_{2}(t,s)\vert g_{2}(s)\,ds\\
 =\max \Bigl\{ 
 \sup_{t\in [0,1/4]}
 \bigl\{   -{\frac {e^2}{768}} \left( -256\,{t}^{3}+64\,{t}^{4}+384\,{t}^{2}-155 \right)  \bigr\},\\
 \sup_{t\in [1/4,3/4]}
  \bigl\{
 -{\frac {e^2}{768}} \left( -256\,{t}^{3}+64\,{t}^{4}+384\,{t}^{2}-155 \right)
 \bigr\}, \\
 \sup_{t\in [3/4,1]}   \bigl\{ 
-{\frac {e^2}{768}} \left( 67-296\,t-256\,{t}^{3}+64\,{t}^{4}+384\,{t}^{2} \right)
 \bigr\} \Bigr\}.
 \end{multline*}{}
We fix $[a_1,b_1]=[a_2,b_2]=[0,1/4]$, obtaining
$$
 \frac{1}{M_{1}} = \inf_{t\in
[0,1/4]}\int_{0}^{1/4} k_1(t,s) g_1(s)\,ds=\inf_{t\in [0,1/4]}\Bigl\{ -{\frac{e^{2}}{3072}} ( -377+256\,{t}^{4}-1024\,{t}^{3}+1536\,{t}^{2} ) \Bigr\},
$$
and
$$
 \frac{1}{M_{2}} = \inf_{t\in
[0,1/4]}\int_{0}^{1/4} k_2(t,s) g_2(s)\,ds=\inf_{t\in [0,1/4]}\Bigl\{ -{\frac {e^2}{3072}}\left( -377+256\,{t}^{4}-1024\,{t}^{3}+1536\,{t}^{2} \right)  \Bigr\}.
$$
By direct computation, we get
$$
c_1=\frac{1}{4};\ m_1=\frac{384}{65e^2} ;\ M_1=\frac{384}{37e^2};\ c_2=\frac{1}{2};\ m_2=\frac{768}{155e^2};\ M_2=\frac{384}{37e^2}.
$$
Let us now consider
$$
f_1(u,v)=\frac{1}{4}(|u|^3+|v|^3+1), \quad f_2(u,v)=\frac{1}{3}(|u|^{\frac{1}{2}}+v^2).
$$
Then, with the choice of $\rho_1=1/6$, $\rho_2=1/3$, $r_1=r_2=1$,  $s_1=3$ and $s_2=5$, we obtain
\begin{align*}
\inf \Bigl\{ f_1(u,v):\; (u,v)\in [0,4\rho_1]\times[-2\rho_2,2\rho_2]
\Bigr\}= f_1(0,0)&>M_1\rho_1, \\
\sup \Bigl\{ f_1(u,v):\; (u,v)\in [-r_1, r_1]\times[-r_2, r_2]\Bigr\}=f_1(1,1)
&< m_1 r_1, \\
\sup \Bigl\{ f_2(u,v):\; (u,v)\in [-r_1, r_1]\times[-r_2, r_2]\Bigr\}=f_2(1,1)
&<m_2 r_2, \\
\inf \Bigl\{ f_1(u,v):\; (u,v)\in [s_1,4s_1]\times[-2s_2,2s_2]\Bigr\}=f_1(s_1,0)
&>M_1 s_1, \\
\inf \Bigl\{ f_2(u,v):\; (u,v)\in [-4s_1,4s_1]\times[s_2, 2s_2]\Bigr\}=f_2(0,s_2)
&>M_2 s_2.
\end{align*}
Thus the conditions $(\mathrm{I}^{0}_{\rho_{1},\rho_2})^{\star}$, $(\mathrm{I}%
^{1}_{r_1,r_{2}})$ and $(\mathrm{I}^{0}_{s_1,s_{2}})$ are satisfied; therefore
the system~\eqref{ellbvpex} has at least two nontrivial solutions.

\end{ex}

\end{document}